\documentclass[12pt,amscd]{amsart}
\usepackage[all]{xy}
\usepackage{graphicx}
\usepackage{amsmath,amsxtra,amssymb,latexsym, amscd,amsthm}
\usepackage{indentfirst}
\usepackage[mathscr]{eucal}

\newtheorem{thm}{Theorem}

\newtheorem{lem}[thm]{Lemma}

\theoremstyle{definition}

\newtheorem{exm}[thm]{Example}

\DeclareMathOperator{\Z}{\mathbb {Z}}
\DeclareMathOperator{\ch}{char}
\DeclareMathOperator{\reg}{reg}
\DeclareMathOperator{\mm}{min-match}

\DeclareMathOperator{\cochord}{cochord}

\begin{document}

\title[Regularity, matchings and Cameron-Walker graphs]{Regularity, matchings and Cameron-Walker graphs}

\author{Tran Nam Trung}
\address{Institute of Mathematics, VAST, 18 Hoang Quoc Viet, Hanoi, Viet Nam, and Institute of Mathematics and Applied Sciences, Thang Long University, Nguyen Xuan Yem road, Hoang Mai district, Hanoi, Viet Nam}

\email{tntrung@math.ac.vn}
\date{}

\dedicatory{}
\commby{}
\begin{abstract} Let $G$ be a simple graph and let $\nu(G)$ be the matching number of $G$. It is well-known that $\reg I(G) \leqslant \nu(G)+1$. In this paper we show that $\reg I(G) = \nu(G)+1$ if and only if every connected component of $G$ is either a pentagon or a Cameron-Walker graph.
\end{abstract}

\maketitle
\section*{Introduction}

Let $G$ be a graph with vertex set $\{1, \ldots, n\}$, and let $R := k[x_1, \ldots , x_n]$ be the polynomial ring over a field $k$. We associated to $G$ an ideal in $R$
$$I(G) = (x_ix_j \mid \{i,j\} \text{ is an edge of } G)$$
which is called the {\it edge ideal} of $G$.

Castelnuovo-Mumford regularity of a homogeneous ideal $I$ in $R$, denoted by $\reg(I)$, is an important algebraic invariant which
measures the complexity of the ideal $I$.  Finding bounds for the regularity of $I(G)$ in terms of combinatorial data of $G$ is an active research program in combinatorial commutative algebra in recent years (see \cite{Ha} and references therein). 

Throughout the paper we assume that $G$ has a least one edge unless otherwise stated. Let $\nu_0(G)$ be the induced matching number of $G$. Katzman \cite{K} showed that
\begin{equation}\label{EQ1}
\reg  (I(G)) \geqslant \nu_0(G)+1.
\end{equation}
There are many classes of graphs $G$ for which the equality occurs (see \cite[Theorem $4.12$]{ABH} for the survey).

For upper bounds, H\`{a} and Van Tuyl \cite{HT} obtained   
\begin{equation}\label{EQ2}
\reg(I(G)) \leqslant \nu(G)+1
\end{equation}
where $\nu(G)$ is the {\it matching number} of $G$. This bound is improved by Woodroofe \cite{W} as follows.  A graph $G$ is {\it chordal} if every induced cycle in $G$ has length $3$, and is {\it co-chordal} if the complement graph $G^c$ of $G$ is chordal. The {\it co-chordal cover number}, denoted $\cochord(G)$, is the minimum number of co-chordal subgraphs required to cover the edges of $G$. Then, 
$$\reg(I(G)) \leqslant \cochord(G)+1.$$

In the paper we interested in graph-theoretically classifying $G$ such that the equality occurs in each bound above. More precisely,

\medskip

\noindent {\bf Problem:} Classify graph-theoretically graphs $G$ such that 
\begin{enumerate}
\item $\reg I(G)  = \nu_0(G)+1$.
\item $\reg I(G)  = \nu(G)+1$.
\item $\reg I(G)  = \cochord(G)+1$.
\end{enumerate}

\medskip

It is worth mentioning that there is a graph $G$ (see Example \ref{E1}) such that the equality $\reg(G) = \nu_0(G)+1$ (resp. $\reg(G) = \cochord(G)+1$) is dependent on the characteristic of the field $k$. Thus we cannot solve Problems $1$ and $3$ without taking into account the characteristic of the based field.

The main result of the paper is to settle Problem $2$. Note that this problem is asked in \cite{ABH}. At first sight when $\nu_0(G) = \nu(G)$, we have $\reg(I(G)) = \nu(G)+1$ by Inequalities $(\ref{EQ1})$ and $(\ref{EQ2})$.  The graph $G$ satisfies $\nu(G)=\nu_0(G)$ is called a Cameron-Walker graph (after Hibi et al. \cite{HHKO}), which is classified in \cite{CaWa, HHKO} as follows.

\begin{thm}{\rm (\cite[Theorem 1]{CaWa} or \cite[p. 258]{HHKO})} \label{CaWa} A connected graph $G$ is Cameron-Walker if and only if it is one of the following graphs (see Figure $1$):
\begin{enumerate}
\item a star;
\item s star triangle; 
\item a graph consisting of a connected bipartite graph with a bipartition partition $(X ,Y)$ such that there is at least one leaf edge attached to each vertex $x \in X$ and that there may be possibly some pendant triangles attached to each vertex $y \in Y$.
\end{enumerate}
\end{thm}

\medskip

\begin{center}

\includegraphics[scale=0.7]{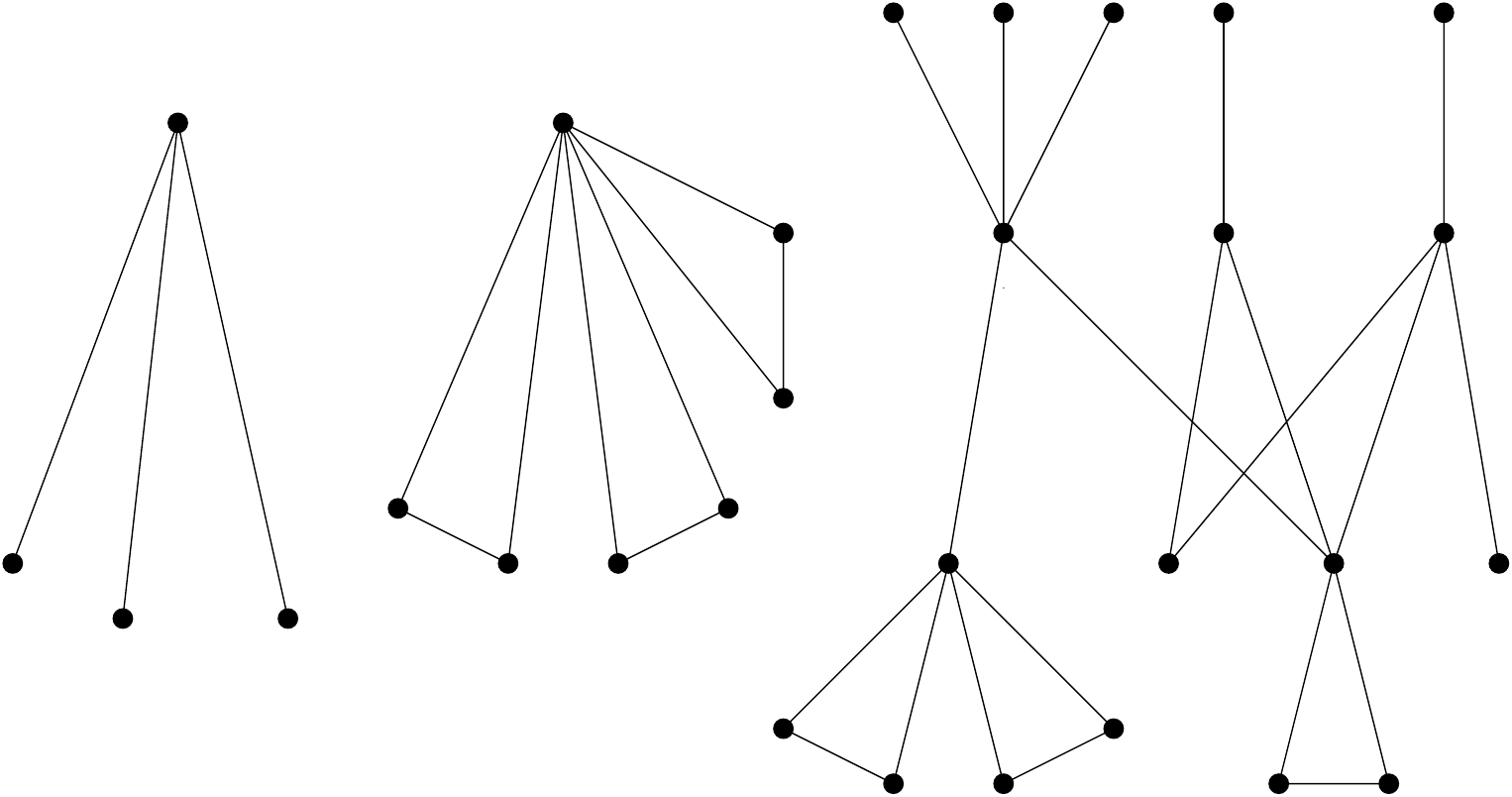}\\
\medskip

{\it Figure $1$. Three kinds of connected Cameron-Walker graphs}
\end{center}

\medskip

 Recall that a pentagon is a cycle of length $5$. Then, the main result of the paper is the following.

\medskip

\noindent {\bf Theorem \ref{main-result}}. {\it Let $G$ be a graph. Then, $\reg(I(G)) = \nu(G)+1$ if and only if each connected component of $G$ is either a pentagon or a Cameron-Walker graph.
}

\medskip

Our paper is structured as follows. In Sect. $1$, we collect notations and terminology used in the paper, and recall a few auxiliary results.
In Sect. $2$, we settle Problem $2$.

\section{Preliminaries}

Let $k$ be a field, and let $R := k[x_1,\ldots,x_n]$ be a standard graded polynomial ring of $n$ variables over $k$. The object of our work is the Castelnuovo-Mumford regularity of graded modules and ideals over $R$. This invariant can be defined via the minimal free resolution. Let $M$ be a finitely generated graded nonzero $R$-module and let
$$0 \rightarrow \bigoplus_{j\in\Z} R(-j)^{\beta_{p,j}(M)} \rightarrow \cdots \rightarrow \bigoplus_{j\in\Z}R(-j)^{\beta_{0,j}(M)}\rightarrow 0$$
be the minimal free resolution of $M$. Then, 
$$\reg(M) = \max\{j-i\mid \beta_{i,j}(M)\ne 0\}.$$

Let $G$ be a finite simple graph. We use the symbols $V(G)$ and $E(G)$ to denote the vertex set and the edge set of $G$, respectively. The algebra-combinatorics framework used in this paper is described via the edge ideal construction. Assume that $V(G)=\{1,\ldots,n\}$. The edge ideal of $G$ is define by
$$I(G) = (x_ix_j \mid \{i,j\} \in E(G)) \subset R.$$

For simplicity, in the sequel, we write $\reg(G)$ to means:
\begin{enumerate}
\item If $G$ has at least one edge, then $\reg(G) = \reg(I(G))$.
\item If $G$ is totally disconnected, then $\reg(G) = 1$.
\item If $G$ is empty, i.e. $V(G)=\emptyset$, then $\reg(G)=0$.
\end{enumerate}

The complementary graph $G^c$ of $G$ is the graph whose vertex set is again $V(G)$ and whose edges are the non-edges of $G$. A graph $G$ is called {\it chordal} if each cycle of length at least $4$ has a chord. We recall the following result of Fr\"{o}berg.

\begin{lem} {\rm (\cite[Theorem 1]{F})} \label{FL4} $\reg(G) = 1$ if and only if $G^c$  is chordal.
\end{lem}

A {\it matching} in a graph $G$ is a subgraph consisting of pairwise disjoint edges.  If the subgraph is an induced subgraph, the matching is an {\it induced matching}. A matching of $G$ is {\it maximal} if it is maximal with respect to inclusion. The {\it matching number} of $G$, denoted $\nu(G)$, is the size of a maximum matching; that is, the maximum number of pairwise disjoint edges; the minimum cardinality of the maximal matchings of $G$ is the {\it minimum matching number} of $G$ and is denoted by $\mm(G)$; and the {\it induced matching number} of  $G$, denoted by $\nu_0(G)$, is  the size of  a maximum induced matching. It follows from  \cite[Theorem $2$]{W} that:

\begin{lem} \label{UB} $\reg(G) \leqslant \mm(G)+1$.
\end{lem}

When there is no confusion, the edge $\{u,v\}$ of $G$ we simply write $uv$. For a vertex $u$ in a graph $G$, let $N_G(u) := \{v \in V(G) \mid uv \in E\}$ be the set of neighbors of $u$, and set $N_G[u] := N_G(u) \cup \{u\}$. An edge $e$ is incident to a vertex $u$ if $u \in e$. The degree of a vertex $u \in V(G)$ , denoted by $\deg_G(u)$, is the number of edges incident to $u$. 

For an edge $e$ in a graph $G$, define $G\setminus e$ to be the subgraph of G with the edge $e$ deleted (but its vertices remained). For a subset $W \subseteq V(G)$ of the vertices in $G$, define $G[W]$ be the induced subgraph of $G$ on $W$ and $G \setminus W$ to be the subgraph of G with the vertices in $W$ (and their incident edges) deleted. When $W = \{u\}$ consists of a single vertex, we write $G\setminus u$ stands for $G\setminus \{u\}$. Define $G_u := G\setminus N_G[u]$. If $e = \{u,v\}$, then $G_e$ to be the subgraph $G\setminus(N_G[u]\cup N_G[v])$ of $G$.

In the study of the regularity of edge ideals, the following lemmas enable us to do induction on the number of vertices and edges.

\begin{lem} \label{FL1} {\rm (\cite[Lemma $3.1$]{Ha})} Let $H$ be an induced subgraph of $G$. Then, $$\reg(H) \leqslant \reg(G).$$
\end{lem}

\begin{lem} \label{FL2} {\rm (\cite[Lemma $2.10$]{DHS})} Let $x$ be a vertex of a graph $G$. Then, 
$$\reg(G) \in \{\reg(G\setminus x), \reg(G_x)+1\}.$$
\end{lem}

\begin{lem}\label{FL3} {\rm (\cite[Theorem $3.5$]{Ha})} Let $e$ be an edge of $G$. Then, 
$$\reg(G) \leqslant \{\reg(G\setminus e), \reg(G_e)+1\}.$$
\end{lem}

\section{Prove the main result}

In this section we classify graphs $G$ that satisfy $\reg(G) = \nu(G)+1$. The following lemma shows that it suffices to consider  connected graphs.

\begin{lem} \label{Comp} Let $G$ be a graph with connected components $G_1,\ldots, G_s$. Then,
\begin{enumerate}
\item $\reg(G) = \sum_{i=1}^s (\reg(G_i)-1) + 1$;
\item $\nu(G) = \sum_{i=1}^s \nu(G_i)$;
\item $\nu_0(G) = \sum_{i=1}^s \nu_0(G_i)$. 
\end{enumerate}
\end{lem}
\begin{proof}
$(1)$ follows from \cite[Corollary 3.10]{ABH}; $(2)$ and $(3)$ are obvious.
\end{proof}

\begin{lem}\label{C1} Let $G$ be a $C_5$-free graph with $\reg(G) = \nu(G)+1$. Then, $\nu(G)=\nu_0(G)$.
\end{lem}
\begin{proof} We prove by induction on $|V(G)|$. If $|V(G)=1|$, then $G$ is just one point, and then $\nu(G) = \nu_0(G)=0$.

Assume that $|V(G)| > 1$. If $\nu(G)=1$, then $\nu_0(G) = 1$, and the lemma follows. 

Assume that $\nu(G) \geqslant 2$. By Lemma \ref{Comp} we may assume that $G$ is connected. If $G$ has a vertex $v$ such that $\reg(G) = \reg(G\setminus v)$. By Lemma \ref{UB} we have $$\nu(G)+1=\reg(G) =\reg(G\setminus v) \leqslant \nu(G\setminus v)+1,$$
hence $\nu(G)\leqslant \nu(G\setminus v)$. The converse inequality $\nu(G)\geqslant \nu(G\setminus v)$ holds since $G\setminus v$ is an induced subgraph of $G$. Thus, $\nu(G\setminus v) = \nu(G)$. It follows that $\reg(G\setminus v) = \nu(G\setminus v)+1$. 

By the induction hypothesis, we have $\nu(G\setminus v) = \nu_0(G\setminus v)$.  Since $G\setminus v$ is an induced subgraph of $G$,  $\nu_0(G\setminus v) \leqslant \nu_0(G)$. Hence, $\nu_0(G) = \nu(G)$, and the lemma holds.

Therefore, by Lemmas \ref{FL1} and \ref{FL2} we may assume that $\reg(G_v) = \reg(G)-1=\nu(G)$ for every $v$.

Let $v$ be a vertex of minimal degree of $G$ and let $x$ be a neighbor of $v$ in $G$.  Since $x$ is not an isolated vertex of $G$, we have $\nu(G_x)\leqslant \nu(G)-1$. Together with equality $\reg(G_x) = \nu(G)$,  it yields $\reg(G_x) \geqslant \nu(G_x)+1$. Together with Lemma \ref{FL2}, we obtain $\reg(G_x) = \nu(G_x)-1$ and $\nu(G)=\nu(G_x)+1$.

By the induction hypothesis we have $\nu(G_x) = \nu_0(G_x)$. Let $m = \nu_0(G_x)$ and $\{e_1,\ldots,e_m\}$ be an induced matching of $G_x$.  Then, $\{xv, e_1,\ldots, e_m\}$ is a maximal matching of $G$. Note that $x$ is not incident to $e_i$ for every $i$. 

Let $S$ be the set of vertices of $G$ which are different from $x$, $v$, and all vertices of  $e_i$ for $i=1,\ldots,m$. Then, $S$ is an independent set of $G$ because $\{xv,e_1,\ldots, e_m\}$ is a maximal matching of $G$.

Assume that $v$ is incident to $e_i$ for some $i$. Without loss of generality we may assume that $i=1$. Let $e_1=yz$ and assume that $v$ is adjacent with $y$.

If $v$ is adjacent with $z$. Let $H := G\setminus\{x,v,y,z\}$. Observe that $\{e_2,\ldots,e_m\}$ is a maximal matching of $H$, so $\reg(H) \leqslant m$ by Lemma \ref{UB}. On the other hand, since $G_v$ is an induced subgraph of $H$, by Lemma \ref{FL1} we have
$\reg(G_v) \leqslant \reg(H) \leqslant m$. Therefore, $\reg(G)=\reg(G_v)+1\leqslant m+1 < \nu(G)+1$, a contradiction. 

\medskip

\begin{center}

\includegraphics[scale=0.7]{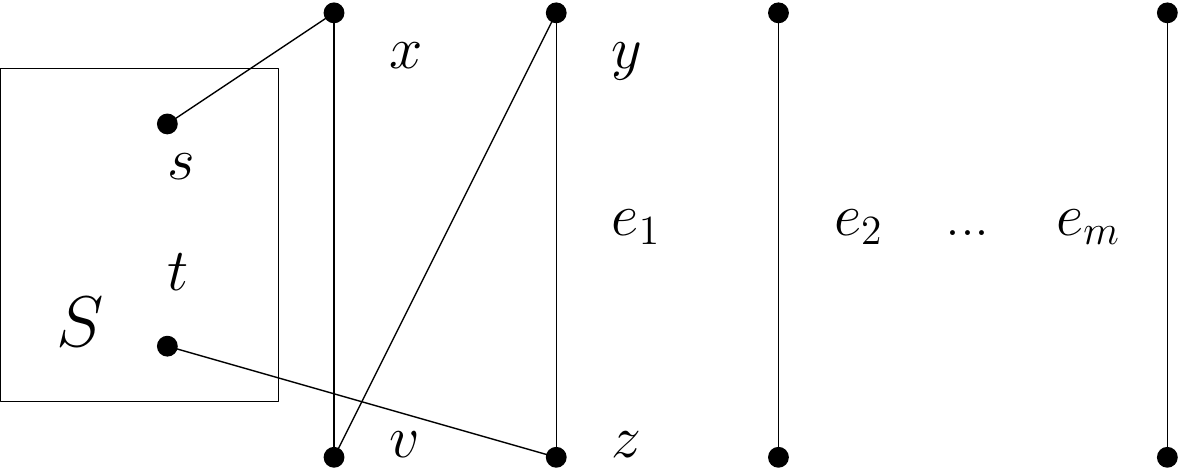}\\
\medskip

{\it Figure $2$.}
\end{center}

Thus, $v$ is not adjacent with $z$ (see Figure $2$). Note that $x$ and $y$ are two neighbors of $v$ so that $\deg_G(v) \geqslant 2$. Since $\deg_G(z)\geqslant \deg_G(v)\geqslant 2$, it follows that $z$ must be incident with some vertex in $S$, say $t$. Then, $t$ is not adjacent with $x$. Because if $t$ is adjacent with $x$, then $G$ would have a cycle $xvyztx$ of length $5$, a contradiction. Similarly, $x$ has a neighbor in $S$, say $s$, which is not adjacent with $z$. In particular, $s\ne t$. But then we would have $\{sx, vy, zt, e_2,\ldots,e_m\}$ is a matching of $G$, so  $\nu(G) \geqslant m+2$, a contradiction.

Therefore, $v$ is not incident to any $e_i$. Then, $\{xv,e_1,\ldots, e_m\}$ is an induced matching of $G$. Since $\nu(G)=m+1$, it follows that $\nu_0(G) = m+1 = \nu(G)$, and the proof of the lemma is complete.
\end{proof}

\begin{lem} \label{C1a} Let $G$ be a graph with $\reg(G) = \nu(G)+1$. If $e$ is an edge of $G$ lying in the middle of a simple path of length $3$ in $G$, then 
$$\reg(G)=\nu(G) = \nu(G\setminus e)=\reg(G\setminus e).$$
\end{lem}

\begin{proof}
Assume that $e=uv$ and $G$ has a simple path $xuvy$. Let $H := G\setminus\{x, u, v, y\}$. We have $G_{uv}$ is an induced subgraph of $H$, so $\reg(G_{uv}) \leqslant \reg(H)$. If $M$ is a matching of $H$, then $M\cup \{xu, vy\}$ is a matching of $G$. It follows that $\nu(H) \leqslant \nu(G)-2$. Therefore, $\reg(G_{uv}) \leqslant \reg(H)\leqslant \nu(H)+1\leqslant \nu(G)-1$. Together with the fact $\reg(G) = \nu(G)+1$, Lemma \ref{FL3} yields $\reg(G) \leqslant \reg(G\setminus uv)$. On the other hand, $\nu(G\setminus uv)\leqslant \nu(G)$ because $G\setminus uv$ is a subgraph of $G$. By Lemma \ref{UB} we obtain
$$\reg(G\setminus uv) \leqslant \nu(G\setminus uv) +1\leqslant \nu(G)+1=\reg(G).$$
It follows that  $\reg(G\setminus uv) = \nu(G\setminus uv)+1=\nu(G)+1=\reg(G)$, as required. 
\end{proof}

\begin{lem}\label{C2} Let $G$ be a connected graph which contains a cycle $C_5$ of length $5$. If $\reg(G) = \nu(G)+1$,  then $G$ is just $C_5$.
\end{lem}
\begin{proof} For simplicity, let $\gamma(G) := |V(G)| + |E(G)|$. Since $C_5$ is a subgraph of $G$, $\gamma(G) \geqslant 10$.

We will prove the lemma by induction on $\gamma(G)$. If $\gamma(G) = 10$, then $G$ is just the cycle $C_5$, and the lemma holds.

Assume that $\gamma(G) \geqslant 10$.  If $V(G) = V(C_5)$, then $G$ is a pentagon with some chords. It follows that $G^c$ is a chordal graph, so $\reg(G) = 2$ by Lemma \ref{FL4}. On the other hand, $\nu(G) = 2$. It implies $\reg(G) < \nu(G)+1$, a contradiction. 

Therefore, $|V(G)| \ne V(C_5)$. We first prove that $G$ has only one cycle. Indeed, if $G$ has another cycle $C \ne C_5$. Since $G$ is connected, it has an edge of $C$, say $e$, such that
\begin{enumerate}
\item $e$ is not in $C_5$;
\item $e$ is in the middle of a simple path of length $3$ in $G$.
\end{enumerate}
By Lemma \ref{C1a} we have $\reg(G\setminus e) = \nu(G\setminus e)$. Note that $G\setminus e$ is connected and has the cycle $C_5$ of length $5$. Since $\gamma(G\setminus e) =\gamma(G)-1$, by the induction hypothesis we have $G\setminus e$ must be $C_5$, a contradiction. Thus, $G$ has only cycle which is just $C_5$.

Now let $uv$ be an edge of $C_5$ and $H := G\setminus uv$. Then, $H$ is a connected graph without cycles, so it is a tree. By Claim $1$ we have
$$\reg(G) = \nu(G) +1 = \nu(H) +1 = \reg(H),$$
so $H$ is a Cameron-Walker graph by Lemma \ref{C1}. 

Since $G=H+uv$ and $G$ has a cycle of length $5$ , $H$ is not a star. Together with Theorem \ref{CaWa}, we conclude that $H$ is a bipartite graph with bipartition $(X,Y)$ such that that every vertex $x$ in $X$ is adjacent to some leaves in $Y$. Let $m= \nu(G)$. Then, we have $m = \nu(H) = |X|$.

Again, because $G = H + uv$ and $G$ has an odd cycle,  we have $u$ and $v$ both are in $X$ or both are in $Y$. If $u,v\in X$, then $G_u$ is an induce subgraph of $G\setminus \{u,v\}$. In this case, $\nu(G_u)\leqslant \nu(G\setminus\{u,v\}) = |X|-2=\nu(G)-2$, and therefore $\reg(G_u) \leqslant \nu(G_u) +1\leqslant \nu(G)-1$. Similarly, $\nu(G\setminus u) = |X|-1=\nu(G)-1$ and $\reg(G\setminus u) \leqslant \nu(G\setminus u)+1 =\nu(G)$. By Lemma \ref{FL2}, we get $\reg(G) \leqslant \nu(G)$, a contradiction.

Therefore, $u,v\in Y$. We may assume that the cycle $C_5$ is $suvtws$. Then, $s,t\in X$ and $w\in Y$. Let $X =\{s,t,x_3,\ldots,x_m\}$. Let $y_3,\ldots,y_m\in Y$ are leaves of $G$ such that $x_iy_i\in E(G)$ for $i=3,\ldots,m$.

We consider two possible cases:

{\it Case $1$}: $\deg_G(w) = 2$. Since $G$ is connected and $V(G)\ne V(C_5)$, there is a vertex $z\in Y\setminus \{u,v,w\}$ such that $z$ is adjacent with $s$ or $t$. We may assume that $z$ is adjacent with $t$ (see Figure $3$). Observe that $z\notin \{y_3,\ldots,y_m\}$, so that 
$$\{zt, ws, uv, x_3y_3,\ldots,x_my_m\}$$
is a matching of $G$. Consequently, $\nu(G) \geqslant m+1$, a contradiction.

\medskip

\begin{center}

\includegraphics[scale=0.7]{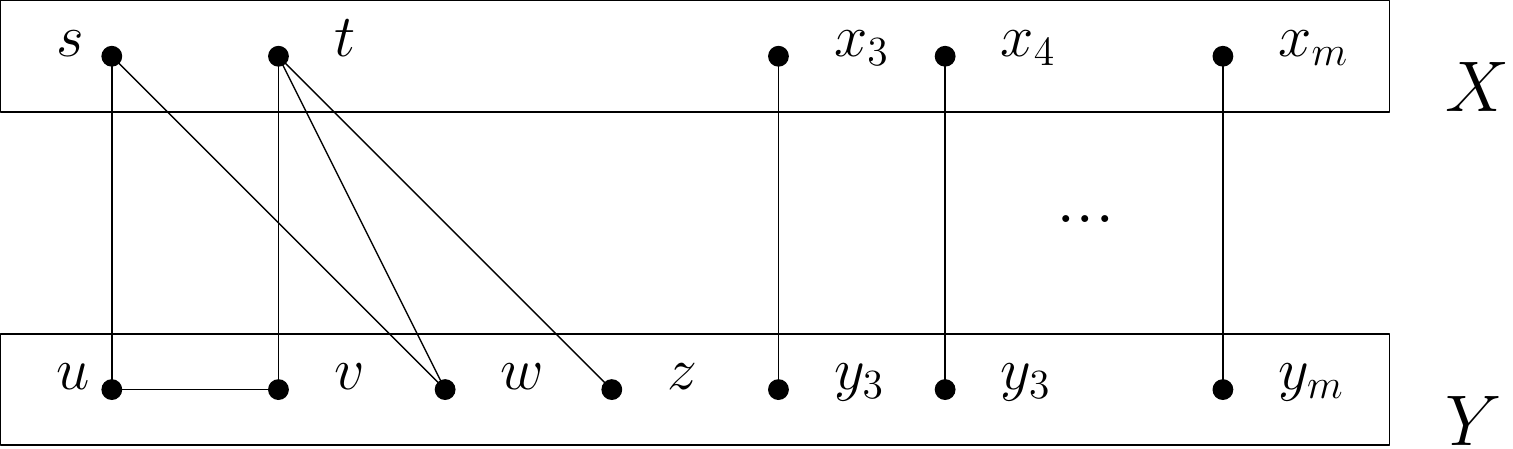}\\
\medskip

{\it Figure $3$.}
\end{center}

{\it Case $2$}: $\deg_G(w) > 2$. We first prove that $u$ and $v$ are two leaves of $H$. Indeed, if it is not the case, we may assume that $v$ is not a leaf. Let $z\in Y$ be a leaf of $H$ that is a neighbor of $t$. It is obvious that $z\notin \{u,v,w\}$. By the same argument as in the case $1$ we obtain a contradiction that $\nu(G) \geqslant m+1$. Thus, $u$ and $v$ are two leaves of $G$. 

We next prove that $\deg_G(s) = \deg_G(t)=2$. Indeed, if its is not the case, we may assume that $\deg_G(t) > 2$. Since $\deg_H(t) = \deg_G(t) > 2$ and $u$ is a leaf of $H$, $t$ has a neighbor in $z\in Y\setminus \{u,v,w\}$. Again by the same argument as in the case $1$ we obtain a contradiction that $\nu(G) \geqslant m+1$. Thus, $\deg_G(s) = \deg_G(t)=2$.

\medskip

\begin{center}

\includegraphics[scale=0.7]{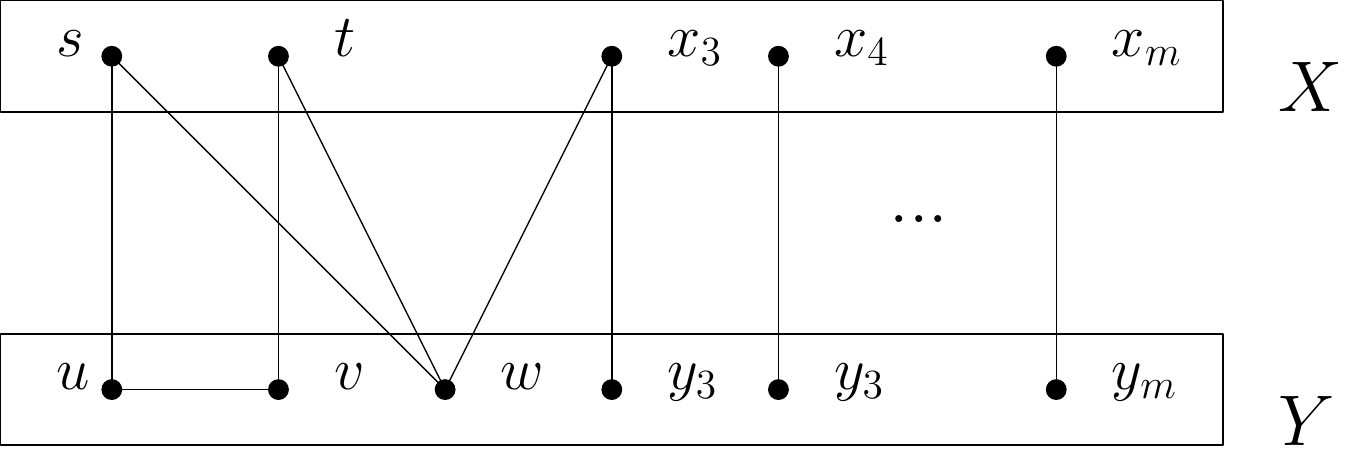}\\
\medskip

{\it Figure $4$.}
\end{center}

Since $\deg_G(w) > 2$, $w$ is adjacent with some vertices in $\{x_3,\ldots,x_m\}$ (see Figure $4$). We may assume that $w$ is adjacent with $x_3,\ldots,x_p$ and not adjacent with $x_{p+1},\ldots,x_m$ for $3\leqslant p \leqslant m$. Let $G'$ be the graph obtained from $G$ by deleting $p-3$ edges $wx_{4},\ldots,wx_p$. By applying successively Lemma \ref{C1a} we have 
$$\reg(G') = \nu(G')+1 = \nu(G)+1=\reg(G).$$
Similarly, if we let $G''$ be the graph obtained from $G'$ by deleting all edges of the form $x_3y$ where $y\in Y\setminus\{w\}$ and $y$ is not a leaf of $G$, then $$\reg(G") = \nu(G")+1 = \nu(G')+1=\reg(G')$$
and then
$$\reg(G") = \nu(G")+1 = \nu(G)+1=\reg(G).$$

Let $S$ be all leaves of $G"$ that are adjacent with $x_3$. Let $G_1 := G[\{s,u,v,t,w,x_3\}\cup S]$ and $G_2 := G\setminus(\{s,u,v,t,w,x_3\} \cup S)$. Then, $G" = G_1 \sqcup G_2$, therefore
$$\reg(G") = \reg(G_1)+\reg(G_2)-1.$$

Since $G_2$ is a Cameron-Walker graph by Theorem $\ref{CaWa}$, $\reg(G_2) = \nu(G_2) + 1 = m-2$. Now we compute $\reg(G_1)$. Observe that $G_1\setminus w$ consists of two connected components that are a path of length $4$ and a star with center $x_3$, so $\reg(G_1\setminus w) = 3$; and $(G_1)_w$ consists of an edge and the set $S$ of isolated vertices, so $\reg((G_1)_w)) = 2$. By Lemma \ref{FL2} we get $\reg(G_2) = 3$. Therefore, 
$$\reg(G")=\reg(G_1)+\reg(G_2)-1 = 3 + (m-2)-1 = m = \nu(G).$$
This contradicts the fact that $\reg(G") = \nu(G)+1$. 

In summary, we must have $V(G) = V(C_5)$, thus $G=C_5$ as we have seen in the beginning of the proof, and thus the lemma follows.
\end{proof}

We are now in position to prove the main result of the paper.

\begin{thm} \label{main-result} Let $G$ be a graph. Then, $\reg(I(G)) = \nu(G)+1$ if and only if each connected component of $G$ is either a pentagon or a Cameron-Walker graph.
\end{thm}
\begin{proof} By Lemma \ref{Comp} we may assume that $G$ is connected. If $G$ is $C_5$-free, then it is a Cameron-Walker graph by Lemma \ref{C1}. If $G$ has a cycle of length $5$, say $C_5$,  it is just $C_5$ by Lemma \ref{C2}, as required.
\end{proof}

We conclude the paper with an example to show that $\reg(G)=\nu_0(G)+1$ (resp. $\reg(G) = \cochord(G)+1$), in general, depends not only the structure of $G$ but also the characteristic of the based field $k$.

\begin{exm}\label{E1} Let $G$ be the graph $G_2$ in \cite[Apendix A]{K}, depicted in Figure $5$.

\medskip

\begin{center}

\includegraphics[scale=0.7]{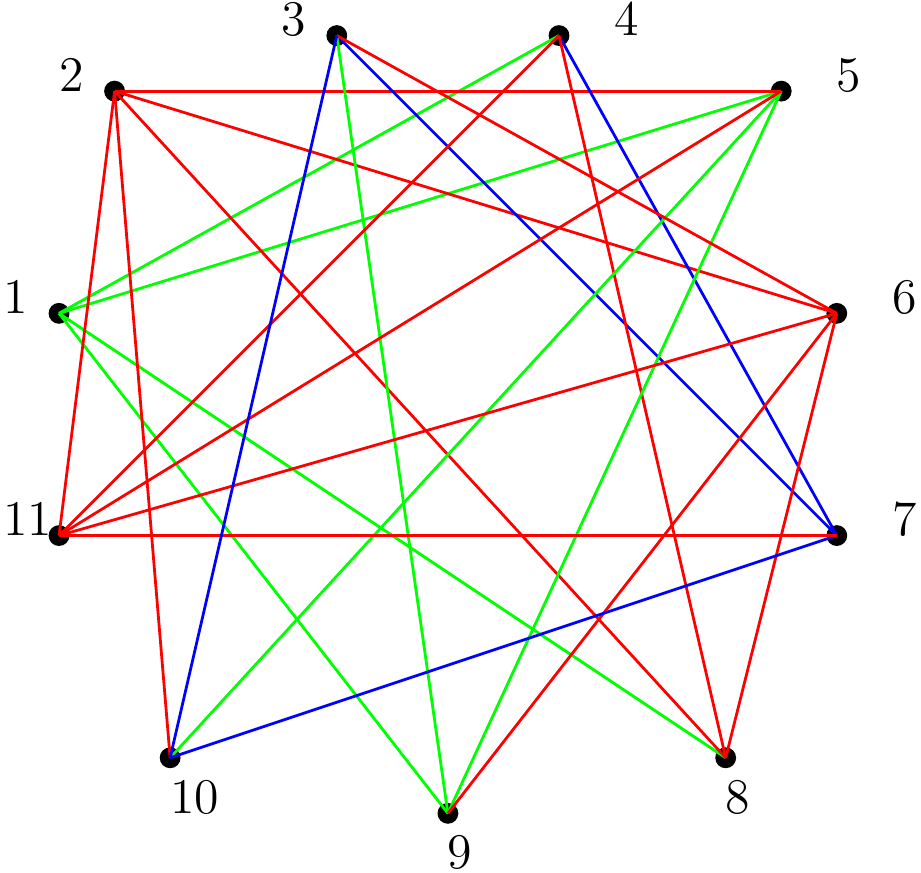}\\
\medskip

{\it Figure $5$.}
\end{center}
\end{exm}

\noindent Then, Macaulay $2$ (see \cite{GS}) computations show that:
\begin{enumerate}
\item If $\ch(k) \ne 2$, then $\reg(G)=3$.
\item If $\ch(k) = 2$, then $\reg(G) = 4$.
\end{enumerate}

We now claim that $\nu_0(G) = 2$ and $\cochord(G) = 3$. Indeed, take $k$ to be a field with $\ch(k)=0$ so that $\reg(G) = 3$. From \cite[Lemma 2.2]{K}, we obtain $\nu_0(G) \leqslant \reg(G)-1 = 2$. Observe that $\{\{1,4\}, \{3,10\}\}$ is an induced matching of $G$, so $\nu_0(G)\geqslant 2$. It follows that $\nu_0(G)=2$.
Next, take $k$ to be a field with $\ch(k)=2$ so that  $\reg(G) = 4$. By Lemma \cite[Theorem $1$]{W}, $\cochord(G) \geqslant \reg(G)-1 = 3$. On the other hand, we have three co-chordal subgraphs $G_1$, $G_2$ and $G_3$ of $G$ that cover the edges of $G$; where we color the edges of $G_1$ by red, the edges of $G_2$ by green, and the edges of $G_3$ by blue. Hence, $\cochord(G) \leqslant 3$, and hence $\cochord(G) = 3$, as claimed.

Thus,
\begin{enumerate}
\item $\reg(G) = \nu_0(G)+1$ if and only if $\ch(k) \ne 2$.
\item $\reg(G) = \cochord(G)+1$ if and only if $\ch(k)=2$.
\end{enumerate}

\subsection*{Acknowledgment}  This work is partially supported by NAFOSTED (Vietnam) under the grant number 101.04 - 2018.307. Part of this work was done while I was at the Vietnam Institute of Advanced Studies in Mathematics (VIASM) in Hanoi, Vietnam. I would like to thank VIASM for its hospitality.

\end{document}